\documentclass[reqno,11pt]{amsart}
\usepackage{multicol, color}
\usepackage{mathrsfs}
\newcommand{\rem}[1]{}

\theoremstyle{plain}
\newtheorem{lemma}{Lemma}
\newtheorem{theorem}[lemma]{Theorem}

\newtheorem{proposition}[lemma]{Proposition}
\newtheorem{definition}[lemma]{Definition}
\theoremstyle{remark}
\newtheorem{remark}{Remark}

\newcommand{\mf}{\mathfrak}

\newcommand*{\ang}[1]{\left\langle #1 \right\rangle}

\def\aa{\alpha}

\begin{document}
\title[Inertial Manifolds for Turbulence Models] {Inertial Manifolds for Certain Sub-Grid Scale $\alpha$-Models of Turbulence}
\date{September 16, 2014}

\author[M. Abu Hamed]{Mohammad Abu Hamed}
\address[M. Abu Hamed]
{Department of Mathematics, Technion - Israel Institute of Technology, Haifa 32000, Israel.
{\bf ALSO}, Department of Mathematics, The College of Sakhnin - Academic College for Teacher Education, Sakhnin 30810, Israel.
} \email{mohammad@tx.technion.ac.il}

\author[Y. Guo]{Yanqiu Guo}
\address[Y. Guo]
{Department of Computer Science and Applied Mathematics \\  Weizmann Institute of Science\\
Rehovot 76100, Israel.} \email{yanqiu.guo@weizmann.ac.il}

\author[E. S. Titi]{Edriss S. Titi}
\address[E. S. Titi]
{Department of Mathematics, Texas A\&M University, 3368 TAMU,
 College Station, TX 77843-3368, USA.  {\bf ALSO},
  Department of Computer Science and Applied Mathematics, Weizmann Institute
  of Science, Rehovot 76100, Israel.} \email{titi@math.tamu.edu and
  edriss.titi@weizmann.ac.il}

\begin{abstract}
 In this note we prove the existence of an inertial manifold, i.e., a global invariant, exponentially attracting, finite-dimensional smooth manifold, for two different sub-grid scale $\alpha$-models of turbulence: the simplified Bardina model and the modified Leray-$\alpha$ model, in two-dimensional space. That is, we show the existence of an exact rule that parameterizes the dynamics of small spatial scales in terms of the dynamics of the large ones. In particular, this implies that the long-time dynamics of these turbulence models is equivalent to that of a finite-dimensional system of ordinary differential equations.
\end{abstract}

\maketitle

\\


\section{INTRODUCTION}

The fidelity of the Navier-Stokes equation (NSE) is in capturing
the dynamics of turbulent flow. However, their downfall is in
reliable direct numerical simulation of turbulence.
Therefore scientists have developed various approximate models
which are computable and preserve some statistical properties of
the physical phenomenon of turbulence, and of particular interest to us in this paper are certain sub-grid scale $\alpha$-models of turbulence.

In many applications, it is enough to capture the mean features of
the flow, to obtain this we need to average the nonlinear term in
the NSE and this leads to the well-known closure problem. In 1980 Bardina et al. \cite{Bardina} introduced a
particular sub-grid scale model which was later simplified by
Layton and Lewandowski (see \cite{Layton-06}) which takes the form:
\begin{align}\label{bar1}
\begin{cases}
v_t-\nu\Delta v+(\bar{v}\cdot \nabla)\bar{v}+\nabla p =f,\\
 \nabla\cdot v=0,\\
  v= \bar{v}-\alpha^{2}\Delta \bar{v}.
\end{cases}
\end{align}
Here the unknowns are the fluid velocity field $v$, and  the ``filtered'' velocity vector $\bar{v}$, as well as the ``filtered'' pressure scalar $p$. In addition, there are two given parameters: $\nu>0$ is the constant kinematic viscosity, and $\alpha>0$ is the length scale parameter which represents the width of the filter. The vector field $f$ is a given body forcing, assumed to be time independent. For more details about model (\ref{bar1}), see \cite{BIL,bardina,LT1,LT2}.

In 2005 Cheskidov-Holm-Olson-Titi \cite{Leray-alpha} introduced the Leray-$\alpha$ model:
\begin{align}\label{Leray}
\begin{cases}
w_t-\nu\Delta w+(\bar{w}\cdot \nabla)w+\nabla p =f,\\
 \nabla\cdot w=0,\\
  w= \bar{w}-\alpha^{2}\Delta \bar{w}.\\
\end{cases}
\end{align}
Leray (1934 \cite{Leray-1934}) established the well-posedness of the NSE in 2D and 3D, by introducing a modified system similar to (\ref{Leray}), for which it was easier to prove the existence and uniqueness of solutions, and then by passing with the parameter $\alpha \rightarrow 0^+$ he achieved the existence of solutions to the NSE. An upper bound of the dimension of the global attractor and an analysis of the energy spectrum of the solutions of the 3D version of (\ref{Leray}) were established in \cite{Leray-alpha}, which suggested that the Leray-$\alpha$ model has great potential to become a good sub-grid scale large-eddy simulation model of turbulence. See also a computational study of this model in \cite{B,KLT,KLTT}.

Inspired by the remarkable performance of the Leray-$\alpha$ model, Ilyin-Lunasin-Titi (2006 \cite{Lunasin-Titi-06}) proposed a modified-Leray-$\alpha$ model:
\begin{align} \label{bard}
\begin{cases}
u_t-\nu\Delta u+(u\cdot \nabla)\bar{u}+\nabla p =f,\\
\nabla\cdot u=0,\\
u= \bar{u}-\alpha^{2}\Delta \bar{u}.
\end{cases}
\end{align}
It was demonstrated in \cite{Lunasin-Titi-06} that the reduced modified-Leray-$\alpha$ model (\ref{bard}) in infinite channels and pipes is equally impressive as a closure model to Reynolds averaged equations as Leray-$\alpha$ model (\ref{Leray}) and other sub-grid scale $\alpha$-models, e.g. the Navier-Stokes-$\alpha$ (also known as the viscous Camassa-Holm equations \cite{NS-alpha-3,NS-alpha-1,NS-alpha-2,FHTP}) and the Clark-$\alpha$ \cite{CHTi}.

Comparing the three turbulence models (\ref{bar1}), (\ref{Leray}) and (\ref{bard}), we see that in the simplified Bardina model (\ref{bar1}), both arguments of the nonlinearity are regularized, while the Leray-$\alpha$ model (\ref{Leray}) regularizes only the first argument of the nonlinear term, i.e. the transport velocity,  and in the modified Leray-$\alpha$ model (\ref{bard}), solely the second argument of the nonlinearity is smoothed, i.e. the transported velocity is regularized.
For the models (\ref{bar1}), (\ref{Leray}) and (\ref{bard}), the global well-posedness in 3D, the existence of a finite dimensional global attractor, and the analysis of their energy spectra have been established in \cite{Bardina,Leray-alpha,Lunasin-Titi-06, Leray-1934}.

Our interest lies in the large-time behavior of the dynamics generated by turbulence models. In particular, we aim to show existence of inertial manifolds for two different systems in 2D: the simplified Bardina model (\ref{bar1}) and the modified-Leray-$\alpha$ model (\ref{bard}), subject to periodic boundary condition, with basic domain $\Omega=[0,2\pi L]^2$.

Long-time behavior of solutions of a large class of dissipative PDEs possesses a resemblance of the behavior of finite-dimensional systems. The concept of inertial manifold was introduced to capture such phenomenon. Indeed, an inertial manifold of an evolution equation is a finite-dimensional Lipschitz invariant manifold attracting \emph{exponentially} all the trajectories of a dynamical system induced by the underlying evolution equation \cite{Foias-85, fsr}. The precise definition is given in section \ref{mani}.
The existence of an inertial manifold for an infinite-dimensional evolution equation represents the best analytical form of reduction of an infinite system to a finite-dimensional one. This is because an
inertial manifold is finite-dimensional, and the restriction of the evolutionary equation to this manifold reduces to a finite system of ODEs, which called the \emph{inertial form} of the given evolutionary equation. As a result, the dynamical properties of the solution of the evolutionary PDE, which is an infinite-dimensional dynamical system can be analyzed by the study of an inertial form which is a finite-dimensional system.

Inertial manifolds were introduced by Foias, Sell and Temam in \cite{Foias-85, fsr}. The idea was employed to a large class of dissipative equations \cite{fst} (see also \cite{Tbook}). A number of dynamical systems possess inertial manifolds, e.g., certain nonlinear reaction-diffusion equations in 2D \cite{CFNT2,fsr,mora} and in 3D \cite{ms}, the Kuramoto-Sivashinsky equation \cite{Foias-88,Foias-85,fst,Tbook}, Cahn-Hilliard equation \cite{CFNT1}, as well as the von K\'arm\'an plate equations \cite{Lasiecka-02}, just to name a few. It is worth mentioning that an original purpose of developing the theory of inertial manifolds was for treating the NSE. Unfortunately, the problem of existence of inertial manifolds for the 2D NSE is still unsolved and we are unaware of any such result for a system of hydrodynamics which does not involve an artificial hyperviscosity. In particular, the question of existence of an inertial manifold is still open even for the 2D Navier-Stokes-$\aa$ model, Leray-$\aa$ model and Clark-$\aa$ model and others.
Recently, the concept of \emph{determine form} was introduced in \cite{FJKT1,FJKT2}, in which it is shown that the long-time dynamics of such models, in particular that of the 2D NSE, is equivalent to the long-time dynamics of an ODE with continuously Lipschitz vector field in certain infinite-dimensional space of trajectories with finite range (see also \cite{JST} for related results).
In this paper, we succeed to obtain the existence of inertial manifolds for the simplified Bardina model (\ref{bar1}) and the modified Leray-$\alpha$ model (\ref{bard}), since the nonlinear terms in these two systems are milder than that of the NSE and other $\alpha$-models of turbulence.

The paper is organized as follows: section \ref{pre} is devoted to the preliminaries and the functional settings.
In section \ref{sec3} and section \ref{sec4}, we study the simplified Bardina model (\ref{bar1}) and the modified Leray-$\alpha$ model (\ref{bard}), respectively, and prove the existence of absorbing balls in various Hilbert spaces, as well as the existence of an inertial manifold for both models. In the appendix, we give a detailed justification of the strong squeezing property for these two systems.

\bigskip

\section{PRELIMINARIES} \label{pre}
We introduce some preliminary background material, which is standard in the mathematical theory of the NSE.
\begin{enumerate}
\item Let $\mathcal{F}$ be the set of all two-dimensional trigonometric vector-valued polynomials
with periodic domain $\Omega$.  We then set
$$
\mathcal{V}=\left\{\phi \in \mathcal{F}:\nabla\cdot\phi = 0 \
\mbox{and} \int_\Omega \phi(x)\ dx = 0\right\}.
$$ We set $H$ and $V$ to be the closures of $\mathcal{V}$ in $L^2_{per}$
and $H^1_{per}$, respectively.

\item We denote by $P_{\sigma}:L^2_{per} \rightarrow H$ the Helmholtz-Leray orthogonal
projection operator, and by $A=-P_{\sigma}\Delta$ the Stokes
operator with the domain $D(A) = (H^2_{per}(\Omega))^2\cap V$. Since we work
with periodic space, then it is known that
\begin{equation*}
Au = -P_{\sigma}\Delta u = -\Delta u, \hspace{.5cm} \mbox{for all
}u \in D(A).
\end{equation*}
The operator $A^{-1}$ is a self-adjoint positive definite compact
operator from $H$ into $H$ (cf. \cite{CF88,TT84}). We denote by
$0<L^{-2}=\lambda_1 \leq \lambda_2 \leq \dots \dots$ the eigenvalues of $A$, repeated according to their multiplicities.

\item We denote by $|\cdot|$ and
$(\cdot,\cdot)$ the $L^{2}_{per}$ norm and the $L^{2}_{per}$ inner product,
respectively. Moreover, one can show that $V=D(A^{1/2})$. Therefore we denote by
$((\cdot,\cdot))=(A^{1/2}\cdot,A^{1/2}\cdot)$, and by
$||\cdot||=|A^{1/2}\cdot|$ the inner product and the norm on $V$,
respectively. We also observe that, $D(A^{s/2}) =(H^s_{per}(\Omega))^2\cap V$ (cf. \cite{CF88,TT84}).
In addition, we denote by $V'$ the dual space of $V$, and by $D(A)'$ the dual space of $D(A)$.

\item For $r<s$, we recall the following version of Poincar\'{e} inequality
\begin{equation}\label{poin}
\lambda_{1}^{s-r}|A^{r}\phi|\leq |A^{s}\phi|,
\end{equation}
for every $\phi\in D(A^{s})$.

\item For $w_{1},w_{2}\in V$, we define the bilinear
form
\begin{equation*}
B(w_{1},w_{2})=P_{\sigma}((w_{1}\cdot\nabla)w_{2}).
\end{equation*}
The bilinear form $B: V\times V \rightarrow V'$ is continuous, and it satisfies
\begin{equation}\label{bilinear}
\langle B(w_1,w_2),w_3 \rangle_{V'} = - \langle B(w_1,w_3),w_2 \rangle_{V'}.
\end{equation}
In particular, $\langle B(w_1,w_2),w_2 \rangle_{V'}=0$. Moreover, $(B(w,w),Aw)=0$ for every $w\in D(A)$ (this is only
true in the 2D periodic case). See \cite{CF88,Tbook,TT84, TT} for proofs. In addition, we shall use the following estimate on the $L^2-$norm of $B(w_1,w_2)$ in 2D:
\begin{align}  \label{in}
|B(w_1,w_2)|\leq c|w_1|^{\frac{1}{2}} \|w_1\|^{\frac{1}{2}} \|w_2\|^{\frac{1}{2}} |A w_2|^{\frac{1}{2}},
\end{align}
which is due to H\"older's inequality and Ladyzhenskaya's inequality in 2D: $|\phi|_{L^4}\leq c|\phi|^{\frac{1}{2}} \|\phi\|^{\frac{1}{2}}$.
\end{enumerate}

Finally, we quote the following classical result (see, e.g., \cite{Tbook,TT84}):
\begin{lemma}
Let $X \subset H \equiv H' \subset X'$ be Hilbert spaces. If $u\in L^2(0,T;X)$ with $u_t \in L^2(0,T;X')$, then $u$ is almost everywhere equal to an absolutely continuous function from $[0,T]$ into $H$ and the following equality holds in the distribution sense on $(0,T)$:
\begin{align} \label{lem}
\frac{d}{dt} |u|_H^2= 2\langle u_t, u \rangle_{X'}.
\end{align}
\end{lemma}

\bigskip

\section{THE SIMPLIFIED BARDINA MODEL} \label{sec3}
This section is devoted to prove the existence of an inertial manifold for the two-dimensional simplified Bardina model.
We apply the Helmholtz-Leray orthogonal projection $P_{\sigma}$ to equation (\ref{bar1}), and obtain the following equivalent functional differential equation (see e.g., \cite{CF88,TT84})
\begin{align}\label{bardina}
\begin{cases}
v_t+\nu Av+B(\bar{v},\bar{v})=f, \\
v=\bar{v}+\alpha^{2}A\bar{v}, \\
v(0)=v_0.
\end{cases}
\end{align}
Moreover, we assume that the forcing term and the initial data have spatial zero mean, i.e., $\int_{\Omega}f(x)dx=\int_{\Omega}{v}_0(x)dx=0$, and hence $\int_{\Omega}{v}(x,t)dx=0$, for all $t\geq 0$.

In \cite{bardina} Cao-Lunasin-Titi proved the global well-posedness of the three-dimensional viscous simplified Bardina model (\ref{bardina}), as well as the existence of a finite-dimensional global attractor. Therefore we will not discuss here the question of well-posedness and the attractor's dimension, because the two-dimensional case follows similar treatment. Notably, it was also shown in \cite{bardina} that the global regularity of the three-dimensional inviscid simplified Bardina model, i.e., when $\nu=0$. In this inviscid case, model (\ref{bardina}) coincides with the inviscid Navier-Stokes-Voigt model, namely, Euler-Voigt model which has been a subject of intensive recent analytical and computational studies (cf. \cite{KT,KLT,LT1,LT2,LRT,RT}).

Now we can quote the following theorem without proof (since it has
been proven in the 3D case in \cite{bardina}) which
states the global existence and uniqueness of regular solutions of
equation (\ref{bardina}).

\begin{theorem}\label{rs}
{\bf(Regular Solution)}  Let $f \in V'$, $v_0 \in V'$, and $T > 0$. Then there exists a unique function $v \in C([0,T];V')\cap L^2([0,T];H)$ with $v_t\in
L^2([0,T];D(A)')$ and  $v(0)=v_0$, and which satisfies (\ref{bardina}) in the following sense:
\begin{equation} \label{W-solnn}
\ang{v_t,w}_{D(A)'} + \nu\ang{Av,w}_{D(A)'}+
\left(B(\bar{v},\bar{v}),w\right) = \langle f,w \rangle_{V'},
\end{equation}
for every $w\in D(A)$. Moreover the solution $v$ depends continuously on the initial data, with respect to the $L^{\infty}([0,T];V')$ norm. Here, equation
(\ref{W-solnn}) is understood in the following sense:
for almost everywhere $t_0,t\in [0,T]$ we have
\begin{equation*}
\ang{v(t), w}_{V'}-\ang{v(t_0),w}_{V'} +
\nu\int_{t_0}^t(v,Aw)+\int_{t_0}^t
\left(B(\bar{v}(s),\bar{v}(s)),w\right) ds = \int_{t_0}^t \langle f,w \rangle_{V'}  ds.
\end{equation*}
\end{theorem}

\smallskip

\subsection{Asymptotic estimates for the long-time dynamics}  \label{bardsec}
This section is devoted to establishing appropriate \emph{a priori} estimates for the long-time dynamics of the solution of (\ref{bardina}).
In particular, we are required to justify the existence of absorbing balls for the
dynamical system induced by equation (\ref{bardina}), in various spaces of functions.
This is needed for our proof for the existence of inertial manifolds.
The estimates provided here are done
formally, but one can prove them rigorously, e.g., by using the Galerkin approximation scheme.
Throughout the following estimates, we assume the forcing $f\in V'$, and the initial data $v(0)\in V'$, thus the corresponding $\bar{v}(0)\in V$.

\subsubsection{$H^{1}$-estimate for $\bar{v}$} \label{sub1}
We take the $D(A)'$ action of equation (\ref{bardina}) on
$\bar{v}$ and use the identities (\ref{bilinear}) and (\ref{lem}) to obtain
\begin{equation}  \label{H1-1}
\frac{1}{2}\frac{d}{dt}(|\bar{v}|^{2}+\alpha^{2}\|\bar{v}\|^{2})+
\nu(\|\bar{v}\|^{2}+\alpha^{2}|A\bar{v}|^{2})=\langle f,\bar{v} \rangle.
\end{equation}
By the Cauchy-Schwarz and Young's inequalities, we have
\begin{equation*}
|\langle f,\bar{v}\rangle|=|(A^{-1}f,A\bar v)|\leq |A^{-1}f||A\bar v|\leq
\frac{|A^{-1}f|^{2}}{2\alpha^{2}\nu}+\frac{\alpha^{2}\nu}{2}|A\bar{v}|^{2}.
\end{equation*}
Consequently, we obtain
\begin{equation*}
\frac{d}{dt}(|\bar{v}|^{2}+\alpha^{2}\|\bar{v}\|^{2})+
\nu(\|\bar{v}\|^{2}+\alpha^{2}|A\bar{v}|^{2})\leq
\frac{|A^{-1}f|^{2}}{\alpha^{2}\nu}.
\end{equation*}

Applying Poincar\'{e} inequality (\ref{poin}) we get
\begin{equation*}
\frac{d}{dt}(|\bar{v}|^{2}+\alpha^{2}\|\bar{v}\|^{2})+
\nu\lambda_{1}(|\bar{v}|^{2}+\alpha^{2}\|\bar{v}\|^{2})\leq
\frac{|A^{-1}f|^{2}}{\alpha^{2}\nu}.
\end{equation*}
We then use Gronwall's inequality to deduce
\begin{equation*}
|\bar{v}(t)|^{2}+\alpha^{2}\|\bar{v}(t)\|^{2}\leq
e^{-\nu\lambda_{1}(t-t_0)}(|\bar{v}(t_{0})|^{2}+\alpha^{2}\|\bar{v}(t_{0})\|^{2})+
\frac{1-e^{-\nu\lambda_{1}(t-t_{0})}}{\alpha^{2}\lambda_{1}\nu^{2}}|A^{-1}f|^{2},
\end{equation*}
for all $t\geq t_{0}\geq 0$. Therefore
\begin{equation*}
\limsup_{t\rightarrow \infty}
(|\bar{v}(t)|^{2}+\alpha^{2}\|\bar{v}(t)\|^{2})\leq
\frac{1}{\alpha^{2}\lambda_{1}\nu^{2}}|A^{-1}f|^{2}.
\end{equation*}
In particular, it follows that
\begin{align*}
\limsup_{t\rightarrow \infty} (1+\alpha^2 \lambda_1) |\bar v(t)|^2 \leq \frac{1}{\alpha^{2}\lambda_{1}\nu^{2}}|A^{-1}f|^{2}
\text{\;\;and\;\;} \limsup_{t\rightarrow \infty} \alpha^{2} \|\bar{v}(t)\|^{2}\leq
\frac{1}{\alpha^{2}\lambda_{1}\nu^{2}}|A^{-1}f|^{2}.
\end{align*}
This immediately implies
\begin{align}
&\limsup_{t\rightarrow \infty}|\bar v(t)|\leq \frac{1}{2}\rho_0:= \left[(1+\alpha^2 \lambda_1)\alpha^{2}\lambda_{1}\nu^{2} \right]^{-\frac{1}{2}}|A^{-1}f| ;   \notag \\
&\limsup_{t\rightarrow \infty} \|\bar v(t)\| \leq \frac{1}{2}\rho_1:=(\alpha^4 \lambda_1 \nu^2)^{-\frac{1}{2}}|A^{-1}f|.    \label{ko}
\end{align}
Thanks to the above, we conclude that, the solution $\bar v(t)$,  after long enough time, enters a ball in $H$, centered at the origin, with radius $\rho_0$. Also, $\bar v(t)$ enters a ball in $V$ with radius $\rho_1$. Notice the growth of $\rho_0$ and $\rho_1$ with respect to the shrinking of $\nu$ satisfies $\rho_0 \sim \nu^{-1}$ and $\rho_1 \sim \nu^{-1}$ asymptotically.

\smallskip

\subsubsection{$H^{2}$-estimate on $\bar{v}$ ($L^2$-estimate on $v$)} \label{sub2}
We take the $D(A)'$ action of equation (\ref{bardina}) on $A\bar{v}$ by using (\ref{lem}), and employ the identity $(B(\bar v,\bar v),A \bar v)=0$ (which is only valid in 2D periodic case, c.f. \cite{CF88,Tbook}). It follows that
\begin{equation*}
\frac{1}{2}\frac{d}{dt}(\|\bar{v}\|^{2}+\alpha^{2}|A\bar{v}|^{2})+
\nu(|A\bar{v}|^{2}+\alpha^{2}|A^{3/2}\bar{v}|^{2})=\langle f,A\bar{v}\rangle.
\end{equation*}

By Cauchy-Schwarz inequality and Young's inequality, we have
\begin{equation*}
|\langle f,A\bar{v}\rangle|=   |(A^{-\frac{1}{2}}f,A^{\frac{3}{2}}\bar{v})| \leq
\frac{|A^{-1/2}f|^{2}}{2\alpha^{2}\nu}+\frac{\alpha^{2}\nu}{2}|A^{3/2}\bar{v}|^{2}.
\end{equation*}
As a result, we reach to
\begin{equation*}
\frac{d}{dt}(\|\bar{v}\|^{2}+\alpha^{2}|A\bar{v}|^{2})+
\nu(|A\bar{v}|^{2}+\alpha^{2}|A^{3/2}\bar{v}|^{2})\leq
\frac{|A^{-1/2}f|^{2}}{\alpha^{2}\nu}.
\end{equation*}
Applying Poincar\'{e} inequality (\ref{poin}) followed by Gronwall's inequality, one has
\begin{equation}\label{Aest}
\|\bar{v}(t)\|^{2}+\alpha^{2}|A\bar{v}(t)|^{2}\leq
e^{-\nu\lambda_{1}(t-t_{0})}(\|\bar{v}(t_{0})\|^{2}+\alpha^{2}|A\bar{v}(t_{0})|^{2})+
\frac{1-e^{-\nu\lambda_{1}(t-t_{0})}}{\alpha^{2}\lambda_{1}\nu^{2}}|A^{-1/2}f|^{2},
\end{equation}
for all $t\geq t_{0}>0$.
Thus,
\begin{equation} \label{limsup2}
\limsup_{t\rightarrow\infty}(\|\bar{v}(t)\|^{2}+\alpha^{2}|A\bar{v}(t)|^{2})\leq \frac{1}{\alpha^{2}\lambda_{1}\nu^{2}}|A^{-1/2}f|^{2}.
\end{equation}
In particular, it follows that
\begin{align*}
&\limsup_{t\rightarrow \infty}\|\bar v(t)\|\leq \frac{1}{2}\tilde \rho_1:= \left[(1+\alpha^2 \lambda_1)\alpha^{2}\lambda_{1}\nu^{2} \right]^{-\frac{1}{2}}|A^{-\frac{1}{2}}f|; \\
&\limsup_{t\rightarrow \infty} |A\bar v(t)| \leq \frac{1}{2}\rho_2:=(\alpha^4 \lambda_1 \nu^2)^{-\frac{1}{2}}|A^{-\frac{1}{2}}f|.
\end{align*}
The above estimate along with (\ref{ko}) shows that $\|\bar v(t)\|\leq \min\{\rho_1,\tilde \rho_1\}$ for sufficiently large time $t$. Also, $\bar v(t)$ enters a ball with radius $\rho_2$ in $D(A)$ after long enough time.

Furthermore, since $v=\bar v+\alpha^2 A \bar v$, one has
$$\limsup_{t\rightarrow \infty} |v(t)| \leq \limsup_{t\rightarrow \infty}|\bar v(t)|+\alpha^2 \limsup_{t\rightarrow \infty}|A\bar v(t)|
\leq (\rho_0+\alpha^2 \rho_2)/2.$$ Thus, after sufficiently large time, $v(t)$ enter a ball in $H$ with the radius $\rho:=\rho_0+\alpha^2 \rho_2.$ Also, note that $\rho \sim \nu^{-1}$ asymptotically.

\smallskip

\subsection{Existence of an inertial manifold} \label{mani}
Denote $R(v):=B(\bar{v},\bar{v})$, then equation (\ref{bardina})
 takes the form
 \begin{align}\label{r}
 \frac{dv}{dt}+\nu Av+R(v)=f,
 \end{align}
 where we assume that $f\in V'$. From the energy estimate in subsection \ref{sub2}, we see that for positive time $t$, one has $\bar v(t)\in D(A)$, and thus $v(t)\in H$ for $t>0$. Moreover, for sufficient large $t$, the solution $v(t)$ enters a ball with radius $\rho$.
 Since we are concerning the large-time behavior of solutions, without loss of generality we can assume $v_0 \in H$, throughout the following discussion.

 Notice that the nonlinear operator $R$ is locally Lipschitz from $H$ to $H$. Indeed, let $v_1$, $v_2\in H$, then the corresponding $\bar v_1$, $\bar v_2\in D(A)$. Furthermore, since $v=\bar v+\alpha^2 A \bar v$, one has $\bar v=(I+\alpha^2 A)^{-1}v$, and thus
\begin{align}  \label{change}
|A \bar v|=|A (I+\alpha^2 A)^{-1} v|  \leq \frac{1}{\alpha^2}|v|.
\end{align}
Then, by using (\ref{in}), along with Poincar\'e inequality and estimate (\ref{change}), we infer
\begin{align}   \label{lips}
|R(v_1)-R(v_2)|&=|B(\bar v_1,\bar v_1)-B(\bar v_2,\bar v_2)|  \notag  \\
&=|B(\bar v_1,\bar v_1-\bar v_2)|+|B(\bar v_1-\bar v_2, \bar v_2)|  \notag\\
&\leq c|\bar v_1|^{\frac{1}{2}} \|\bar v_1\|^{\frac{1}{2}} \|\bar v_1-\bar v_2 \|^{\frac{1}{2}} |A \bar v_1-A \bar v_2|^{\frac{1}{2}}
+c|\bar v_1-\bar v_2|^{\frac{1}{2}}\|\bar v_1-\bar v_2\|^{\frac{1}{2}} \|\bar v_2\|^{\frac{1}{2}} |A \bar v_2|^{\frac{1}{2}} \notag \\
&\leq c\lambda_1^{-1}  (|A\bar v_1|+|A\bar v_2|) |A\bar v_1-A \bar v_2|   \notag\\
&\leq c\lambda_1^{-1} \alpha^{-4} (|v_1|+|v_2|) |v_1-v_2|.
\end{align}

 As in \cite{CF88,fsr,fst,Tbook}, in order to avoid certain technical difficulties for large values of $|v|$, resulting from the nonlinearity, we truncate the nonlinear term by a smooth cutoff function outside the ball of radius $2\rho$ in $H$. Indeed, let $\theta: \mathbb R^+ \rightarrow [0,1]$ with $\theta(s)=1$ for $0\leq s \leq 1$, $\theta(s)=0$ for $s\geq 2$, and $|\theta'(s)|\leq 2$ for $s\geq 0$. Define $\theta_{\rho}(s)=\theta(s/\rho)$, for $s\geq 0$. We consider the following ``prepared" equation, which is a modification of (\ref{r}):
 \begin{equation} \label{rr}
 \frac{dv}{dt}+\nu Av+\theta_{\rho}(|v|)(R(v)-f)=0.
 \end{equation}
 Notice that (\ref{r}) and (\ref{rr}) have the same asymptotic behaviors in time, and the same dynamics in the neighborhood of the global attractor. This is because we have shown that for $t$ sufficiently large, $v(t)$ enters a ball in $H$ with radius $\rho$.
 On the other hand, the advantage of (\ref{rr}) compared to (\ref{r}) is that (\ref{rr}) possesses an absorbing \emph{invariant} ball in $H$. To see this, take the scalar product of (\ref{rr}) with $v$, and then for $|v|\geq 2\rho$, one has
\begin{align*}
\frac{1}{2}\frac{d}{dt}|v|^2 + \lambda_1 \nu |v|^2 \leq \frac{1}{2}\frac{d}{dt}|v|^2 +  \nu \|v\|^2 =0,
\end{align*}
since $\theta_{\rho}(|v|)=0$ for $|v|\geq 2\rho$. It follows that, if $|v_0| > 2\rho$, the orbit of the solution to (\ref{rr}) will converge exponentially to the ball of radius 2$\rho$ in $H$, while if $|v_0|\leq 2\rho$, the solution does not leave this ball.

Furthermore, since $R: H\rightarrow H$ is locally Lipschitz, the truncated nonlinearity $F(v):=\theta_{\rho}(|v|)R(v)$ is \emph{globally} Lipschitz from $H$ to $H$. To see this, we let $v_1$, $v_2\in H$, and calculate for three cases:
\begin{enumerate}
\item if $|v_1|\geq 2\rho$ and $|v_2|\geq 2\rho$, then $F(v_1)=F(v_2)=0$;
\item if $|v_1|\geq 2\rho \geq |v_2|$, then $\theta_{\rho}(|v_1|)=0$, thus
\begin{align*}
|F(v_1)-F(v_2)|&=|\theta_{\rho}(|v_1|)R(v_2)-\theta_{\rho}(|v_2|)R(v_2)| \\
&\leq \frac{2}{\rho}|v_1-v_2| |R(v_2)|
\leq c\rho  \lambda^{-1} \alpha^{-4} |v_1-v_2|,
\end{align*}
by virtue of (\ref{lips}) and the property of $\theta$.
\item if $|v_1|\leq 2\rho$ and $|v_2|\leq 2\rho$, then
\begin{align*}
|F(v_1)-F(v_2)|
&\leq |\theta_{\rho}(|v_1|)(R(v_1)-R(v_2))|+|R(v_2)(\theta_{\rho}(|v_1|)-  \theta_{\rho}(|v_2|)    )    | \\
&\leq c \rho \lambda_1^{-1} \alpha^{-4} |v_1-v_2|,
\end{align*}
due to (\ref{lips}) and the property of $\theta$.
\end{enumerate}
A summary of these three cases yields
\begin{align} \label{sum}
|F(v_1)-F(v_2)|\leq  \mathscr L |v_1-v_2|, \text{\;\;where\;\;}  \mathscr L:=c\rho \lambda_1^{-1} \alpha^{-4}.
\end{align}

Since the nonlinearity of (\ref{rr}) is globally Lipschitz, we shall see that equation (\ref{rr}) possesses the \emph{strong squeezing property} stated in Proposition \ref{prop1}, provided certain spectral gap condition is fulfilled. Indeed, for $\gamma>0$ and $n\in \mathbb N$, we define the cone
\begin{align}  \label{cone}
\Gamma_{n,\gamma}:=\left\{ \begin{pmatrix}v_1\\v_2\end{pmatrix} \in H \times H: |Q_n(v_1-v_2)| \leq \gamma |P_n(v_1-v_2)|   \right \}.
\end{align}
The strong squeezing property asserts: if the dynamics of two trajectories starts inside the cone $\Gamma_{n,\gamma}$, then the trajectories stay inside the cone forever, and the higher Fourier modes of the difference are dominated by the lower modes (i.e. \emph{the cone invariance property}); on the other hand, for as long as the two trajectories are outside the cone, then the higher Fourier modes of the difference decay exponentially fast (i.e. \emph{the decay property}). More precisely, we have the following result.

\begin{proposition} \label{prop1}
Let $v_1$ and $v_2$ be two solutions of (\ref{rr}). Then (\ref{rr}) satisfies the following properties:
\begin{enumerate}
\item The \emph{cone invariance property}: Assume that $n$ is large enough such that the spectral gap condition $\lambda_{n+1}-\lambda_n > \frac{\mathscr L (\gamma+1)^2}{\nu \gamma} $ holds. If $\begin{pmatrix}v_1(t_0)\\v_2(t_0)\end{pmatrix}\in \Gamma_{n,\gamma} $ for some $t_0\geq 0$,
then $\begin{pmatrix}v_1(t)\\v_2(t)\end{pmatrix}\in \Gamma_{n,\gamma} $ for all $t\geq t_0$;

\item The \emph{decay property}: Assume that $n$ is large enough such that $\lambda_{n+1}>\nu^{-1}\mathscr L\left(\frac{1}{\gamma}+1\right)$. If $\begin{pmatrix}v_1(t)\\v_2(t)\end{pmatrix} \not\in \Gamma_{n,\gamma} $ for $0\leq t \leq  T$, then
\begin{align*}
|Q_n (v_1(t)-v_2(t))| \leq |Q_n(v_1(0))-v_2(0)|e^{-b_n t}, \text{\;\;for\;\;} 0\leq t\leq T,
\end{align*}
where $b_n  := \nu \lambda_{n+1}-\mathscr L\left(\frac{1}{\gamma}+1\right)>0$.
\end{enumerate}
\end{proposition}

\begin{proof}
See the appendix.
\end{proof}

Notice that, the eigenvalues of the operator $A$ satisfies the spectral gap condition:
\begin{align}   \label{gap}
\limsup_{j\rightarrow\infty}(\lambda_{j+1}-\lambda_{j})=\infty.
\end{align}
Indeed, since the eigenvalues of $A$ in the periodic domain are of the form $L^{-2}(k_{1}^{2}+k_{2}^{2})$,
the spectral gap condition (\ref{gap}) is available due to a classical result in number theory:
\begin{theorem}{\emph{(Richards \cite{Ri})}}  \label{rich}
The sequence $\{\gamma_{k}=m_1^{2}+m_2^{2}: m_1,m_2\in\mathbb{Z} \ \ \text{and}
\ \ \gamma_{k+1}\geq \gamma_{k}\}$ has a subsequence
$\{\gamma_{k_{j}}\}$ such that $\gamma_{k_{j+1}}-\gamma_{k_{j}}\geq
\delta\log(\gamma_{k_{j}})$ for some $\delta>0$.
\end{theorem}

Obviously, (\ref{gap}) implies the required condition in Proposition \ref{prop1}, i.e., there exists $n\in \mathbb N$ such that
 $\lambda_{n+1}-\lambda_n > \frac{4 \mathscr L}{\nu}$ and $\lambda_{n+1}>\nu^{-1}\mathscr L\left(\frac{1}{\gamma}+1\right)$, and thus for such $n$ large enough, the strong squeezing property holds for the ``prepared" equation (\ref{rr}).

\begin{definition} \emph{(\bf{Inertial Manifold})} \cite{fsr}
Consider the solution operator $S(t)$ generated by the ``prepared" equation (\ref{rr}). A subset $\mathcal M\in H$ is called an \emph{initial manifold} for (\ref{rr}) if the following properties are satisfied :
\begin{enumerate}
\item $\mathcal M$ is a finite-dimensional Lipschitz manifold;
\item $\mathcal M$ is invariant, i.e. $S(t) \mathcal M\subset \mathcal M$, for
all $t\geq 0$;
\item\label{ae} $\mathcal M$ attracts exponentially all the solutions of (\ref{rr}).
\end{enumerate}
\end{definition}
Clearly, property (\ref{ae}) implies that $\mathcal M$ contains the global attractor.

Next, we state a fundamental theorem concerning that the strong squeezing property implies the existence of an inertial manifold and the \emph{exponential tracking} (cf. \cite{fst}) for dissipative evolution equations. There are several proofs of this theorem that can be found in \cite{CFNT1,Foias-88,fst,Robin,Tbook}.

\begin{theorem} \label{thm-g}
Consider a nonlinear evolutionary equation of the type $v_t+A v+ N(v)=0$, where $A$ is a linear, unbounded self-adjoint positive operator, acting in a Hilbert space $H$, such that $A^{-1}$ is compact, and $N: H\rightarrow H$ is a nonlinear operator. Assume the solution $v(t)$ enters a ball in $H$ with the radius $\rho$ for sufficiently large time $t$. For $\gamma>0$ and $n\in \mathbb N$, we define the cone $\Gamma_{n,\gamma}$ in (\ref{cone}). Assume there exists $n\in \mathbb N$ such that the ``prepared" equation  $v_t+A v+ \theta_{\rho}(|v|)N(v)=0$ satisfies the strong squeezing property, i.e., for any two solutions $v_1$ and $v_2$ of the ``prepared" equation,
\begin{itemize}
\item  if $\begin{pmatrix}v_1(t_0)\\v_2(t_0)\end{pmatrix}\in \Gamma_{n,\gamma} $ for some $t_0\geq 0$,
then $\begin{pmatrix}v_1(t)\\v_2(t)\end{pmatrix}\in \Gamma_{n,\gamma} $ for all $t\geq t_0$;
\item if $\begin{pmatrix}v_1(t)\\v_2(t)\end{pmatrix} \not\in \Gamma_{n,\gamma} $ for $0\leq t \leq  T$, then there exists $a_n>0$ such that
\begin{align*}
|Q_n (v_1(t)-v_2(t))|_H \leq e^{-a_n t} |Q_n(v_1(0))-v_2(0)|_H , \text{\;\;for\;\;} 0\leq t\leq T.
\end{align*}
\end{itemize}
Then the ``prepared" equation possesses an $n$-dimensional inertial manifold in $H$. In addition, the following exponential tracking property holds: for any $v_0\in H$, there exists a time $\tau\geq 0$ and a solution $S(t)\varphi_0$ on the inertial manifold such that
\begin{align*}
|S(t+\tau)v_0-S(t)\varphi_0|_H \leq C e^{-a_n t},
\end{align*}
where the constant $C$ depends on $|S(\tau)v_0|_H$ and $|\varphi_0|_H$.
\end{theorem}

Since we have shown that the strong squeezing property holds for (\ref{rr}) provided $n$ is large enough, by using Theorem \ref{thm-g}, we obtain the following result for the simplified Bardina model.
\begin{theorem}
The ``prepared" equation (\ref{rr}) of the simplified Bardina model possesses an $n$-dimensional inertial manifold $\mathcal M$ in $H$, i.e., the solution $S(t)v_0$ of (\ref{rr}) approaches the invariant Lipschitz manifold $\mathcal M$ exponentially. Furthermore, the following exponential tracking property holds: for any $v_0\in H$, there exists a time $\tau\geq 0$ and a solution $S(t)\varphi_0$ on the inertial manifold $\mathcal M$ such that
\begin{align*}
|S(t+\tau)v_0-S(t)\varphi_0| \leq C e^{-b_n t},
\end{align*}
where $b_n$ is defined in Proposition \ref{prop1}, and the constant $C$ depends on $|S(\tau)v_0|$ and $|\varphi_0|$.
\end{theorem}

\bigskip

\section{MODIFIED-LERAY-$\alpha$ MODEL}  \label{sec4}
This section is devoted to proving the existence of an inertial manifold for the modified-Leray-$\alpha$ model (\ref{bard}).
Applying the Helmholtz-Leray orthogonal projection $P_{\sigma}$ to (\ref{bard}), we obtain the following equivalent functional differential equation:
\begin{align}\label{bardina2}
\begin{cases}
u_t+\nu Au+B(u,\bar{u})=f \\
u=\bar{u}+\alpha^{2}A\bar{u} \\
u(x,0)=u_0(x).
\end{cases}
\end{align}
An analytical study of the modified-Leray-$\alpha$ model has been presented in \cite{Lunasin-Titi-06}. Specifically, it was shown that (\ref{bardina2}) is globally well-posed in 3D. In addition, an upper bound for the dimension of its global attractor and analysis of the energy spectrum were established. The proof of global well-posedness in 2D is very similar, so we just state the result and omit its proof.
\begin{theorem}\label{rs}
{\bf(Regular Solution)}  Let $f \in H$, $u_0 \in V'$, and $T > 0$. Then there exists a unique function $u \in C([0,T];V')\cap L^2([0,T];H)$ with $u_t\in
L^2([0,T];D(A)')$ and $u(0)=u_0$, and which satisfies equation (\ref{bardina2}) in the following sense:
\begin{equation}\label{W-soln}
\ang{\frac{du}{dt},w}_{D(A)'} + \nu\ang{Au,w}_{D(A)'}+
\left(B(u,\bar{u}),w\right) = \langle f,w \rangle_{V'},
\end{equation}
for every $w\in D(A)$. Moreover the solution $v$ depends continually on the initial data with respect to the $L^{\infty}([0,T];V')$ norm. Here, the equation
(\ref{W-soln}) is understood in the following sense:
for almost everywhere $t_0,t\in [0,T]$ we have
\begin{equation*}
\ang{u(t), w}_{V'}-\ang{u(t_0),w}_{V'} +
\nu\int_{t_0}^t(u,Aw)+\int_{t_0}^t
\left(B(u(s),\bar{u}(s)),w\right) ds = \int_{t_0}^t \langle f,w \rangle_{V'}  ds.
\end{equation*}
\end{theorem}

\smallskip

\subsection{Asymptotic estimates for the long-time dynamics} \label{modsec}
In order to prove the existence of an inertial manifold, it is required to establish
appropriate \emph{a priori} estimates on the long-time dynamics of the solution.
In particular, we are required to find absorbing balls for the
dynamical system induced by the equation (\ref{bardina2}) in
various spaces of functions. The estimates provided here are done
formally, but can be justified rigorously, for instance, by using the standard Galerkin approximation method.
During our estimates, $u_0\in V'$ and $f\in H$.

\subsubsection{$H^1$-estimate on $\bar u$} \label{sub3}
Taking the $D(A)'$ action of the equation (\ref{bardina2}) on $\bar u$ by using the fact $(B(u,\bar u),\bar u)=0$ and (\ref{lem}), we obtain
\begin{align}  \label{L2-1}
\frac{1}{2}\frac{d}{dt}(|\bar{u}|^{2}+\alpha^{2}\|\bar{u}\|^{2})+
\nu(\|\bar{u}\|^{2}+\alpha^{2}|A\bar{u}|^{2})=(f,\bar{u}).
\end{align}

Notice that, the energy identity (\ref{L2-1}) is almost identical to (\ref{H1-1}) from the analysis of the simplified Bardina model. Therefore, we can adopt the estimate in the subsection \ref{sub1} to conclude
\begin{align*}
&\limsup_{t\rightarrow \infty}|\bar u(t)|\leq \frac{1}{2}\rho_0:= \left[(1+\alpha^2 \lambda_1)\alpha^{2}\lambda_{1}\nu^{2} \right]^{-\frac{1}{2}}|A^{-1}f| ;   \\
&\limsup_{t\rightarrow \infty} \|\bar u(t)\| \leq \frac{1}{2}\rho_1:=(\alpha^4 \lambda_1 \nu^2)^{-\frac{1}{2}}|A^{-1}f|.
\end{align*}
From this, we conclude that, the solution $\bar u(t)$,  after a sufficiently large time, enters a ball in $H$ with radius $\rho_0$, and also enters a ball in $V$ with radius $\rho_1$. In addition the growth of the radii $\rho_0$ and $\rho_1$ with respect to the shrinking of the viscosity $\nu$ satisfies $\rho_0 \sim \nu^{-1}$ and $\rho_1 \sim  \nu^{-1}$.

\subsubsection{$L^2$-estimate on $u$ ($H^2$-estimate on $\bar u$)}
By taking the $D(A)'$ action of the equation (\ref{bardina2}) on $u$ and using (\ref{lem}), we have
\begin{align*}
\frac{1}{2} \frac{d}{dt}|u|^2 +  \nu \|u\|^2+(B(u,\bar u),u)=(f,u).
\end{align*}
Recall in subsection \ref{sub2} when we derived $L^2$-estimate on $v$ ($H^2$-estimate on $\bar v$) for the simplified Bardina model, we used the identity $(B(\bar v,\bar v),A\bar v)=0$ (in the periodic 2D case) to eliminate the nonlinearity. On the other hand, for the NSE, the $L^2$-estimate is fairly easy, since $(B(u,u),u)=0$. However, under the current situation, the nonlinear term $(B(u,\bar u),u)$ does not vanish, which causes the estimate to be slightly more involved. Indeed, by using H\"older's inequality, and the Ladyzhenskaya inequality $|u|_{L^4} \leq c|u|^{\frac{1}{2}} \|u\|^{\frac{1}{2}}$, as well as the Young's inequality, we infer
\begin{align*}
|(B(u,\bar u),u)| \leq |u|_{L^4}^2 \|\bar u\|
\leq c |u| \|u\|  \|\bar u\|
\leq \frac{\nu}{4} \|u\|^2 + \frac{c}{\nu} |u|^2 \|\bar u\|^2.
\end{align*}
Also, $|(f,u)|=|(A^{-\frac{1}{2}}f, A^{\frac{1}{2}}u)| \leq |A^{-\frac{1}{2}}f| \|u\|\leq \frac{\nu}{4} \|u\|^2 + \frac{1}{\nu} |A^{-\frac{1}{2}}f|^2$. Combining the above estimates, we obtain
\begin{align*}
\frac{d}{dt}|u|^2 +  \nu \|u\|^2 \leq \frac{c}{\nu}|u|^2 \|\bar u\|^2+ \frac{2}{\nu}|A^{-\frac{1}{2}}f|^2 .
\end{align*}

In subsection \ref{sub3}, we have shown that there exists $t_1>0$ such that $|\bar u(t)|\leq \rho_0$ and $\|\bar u(t)\|\leq \rho_1$ provided $t\geq t_1$. As a result,
\begin{align}   \label{L2-2}
\frac{d}{dt}|u|^2 +  \nu \|u\|^2  \leq \frac{c}{\nu} \rho_1^2 |u|^2+ \frac{2}{\nu}|A^{-\frac{1}{2}}f|^2, \text{\;\;for all\;\;} t\geq t_1.
\end{align}

We attempt to derive a uniform bound for $|u(t)|$. To this end, we integrate between $s$ and $t+\frac{1}{\nu \lambda_1}$ for $t_1\leq t\leq s\leq t+\frac{1}{\nu \lambda_1}$:
\begin{align*}
|u(t+\frac{1}{\nu \lambda_1} )|^2\leq |u(s)|^2 + \frac{c}{\nu} \rho_1^2 \int_t^{t+\frac{1}{\nu \lambda_1}} |u(s)|^2 ds + \frac{2}{\nu^2 \lambda_1}|A^{-\frac{1}{2}}f|^2.
\end{align*}
Then, integrating with respect to $s$ from $t$ to $t+\frac{1}{\nu \lambda_1} $ gives
\begin{align} \label{L2-4}
\frac{1}{\nu \lambda_1}|u(t+\frac{1}{\nu \lambda_1})|^2
\leq \left(\frac{c}{\nu^2 \lambda_1} \rho_1^2+1\right) \int_t^{t+\frac{1}{\nu \lambda_1}} |u(s)|^2 ds + \frac{2}{\nu^3 \lambda_1^2}|A^{-\frac{1}{2}}f|^2, \text{\;\;for all\;\;} t\geq t_1.
\end{align}
In order to control the right-hand side, we should obtain a bound on $\int_t^{t+\frac{1}{\nu \lambda_1}} |u(s)|^2 ds$. To this end, we deduce from (\ref{L2-1}) by using Cauchy-Schwarz and Young's inequalities:
\begin{equation*}
\frac{d}{dt}(|\bar{u}|^{2}+\alpha^{2}||\bar{u}||^{2})+
\nu(||\bar{u}||^{2}+\alpha^{2}|A\bar{u}|^{2})\leq
\frac{|A^{-1}f|^{2}}{\alpha^{2}\nu}.
\end{equation*}
Integrating the above inequality from $t$ to $t+\frac{1}{\nu \lambda_1}$ yields
\begin{align*}
\nu \alpha^2 \int_t^{t+\frac{1}{\nu \lambda_1}} |A\bar{u}(s)|^{2} ds &\leq |\bar{u}(t)|^{2}+\alpha^{2}||\bar{u}(t)||^{2}+\frac{|A^{-1}f|^{2}}{\alpha^{2}\nu^2 \lambda_1} \\
&\leq \rho_0^2+\alpha^2 \rho_1^2+\frac{|A^{-1}f|^{2}}{\alpha^{2}\nu^2 \lambda_1}, \text{\;\;for\;\;} t\geq t_1,
\end{align*}
where we have used the fact that $|\bar u(t)|\leq \rho_0$ and $\|\bar u(t)\|\leq \rho_1$ for $t\geq t_1$.

By definition $u=\bar u+\alpha^2 A \bar u$, it follows that $|u|^2\leq 2(|\bar u|^2+\alpha^4 |A\bar u|^2)\leq 2\left(\frac{1}{\lambda_1^2}+\alpha^4 \right) |A\bar u|^2 $ due to Poincar\'e inequality. Consequently, for $t\geq t_1$, one has
\begin{align} \label{L2-3}
\int_t^{t+\frac{1}{\nu \lambda_1}}|u(s)|^2 ds &\leq 2\left(\frac{1}{\lambda_1^2}+\alpha^4 \right)  \int_t^{t+\frac{1}{\nu \lambda_1}} |A\bar u(s)|^2 ds \notag\\
&\leq C_0:=\left(\frac{1}{\lambda_1^2}+\alpha^4\right )\frac{2}{\nu \alpha^2}\left(\rho_0^2+\alpha^2 \rho_1^2+\frac{|A^{-1}f|^{2}}{\alpha^{2}\nu^2 \lambda_1}\right).
\end{align}
Substituting (\ref{L2-3}) into (\ref{L2-4}), we conclude
\begin{align*}
|u(t+\frac{1}{\nu \lambda_1} ) |^2 \leq \rho_3^2:= \left(\frac{c}{\nu} \rho_1^2+\nu \lambda_1\right) C_0+\frac{2}{\nu^2 \lambda_1}|A^{-\frac{1}{2}} f |^2  ,  \text{\;\;for\;\;} t\geq t_1.
\end{align*}
This indicates that, for $t\geq t_1+\frac{1}{\nu \lambda_1} $, the solution $u(t)$ enters a ball in $H$ with the radius $\rho_3$.

Furthermore, the growth of the radius $\rho_3$ with respect to the shrinking of the viscosity $\nu$ satisfies $\rho_3 \sim \nu^{-3}$.

\smallskip

\subsubsection{$H^1$-estimate on $u$} \label{sec5}
We take the $D(A)'$ action of the equation (\ref{bardina2}) on $Au$. It follows from (\ref{lem}) that
\begin{align*}
\frac{1}{2}\frac{d}{dt}\|u\|^2 +\nu |Au|^2 + (B(u,\bar u), Au)=(f,Au).
\end{align*}
By using Cauchy-Schwarz and Young's inequalities, one has
\begin{align*}
\frac{d}{dt}\|u\|^2 +\nu |Au|^2 \leq \frac{2}{\nu}(|B(u,\bar u)|^2+|f|^2).
\end{align*}
Recall we have shown that $\|\bar u(t)\|\leq \rho_1$ for $t\geq t_1$, as well as $|u(t)|\leq \rho_3$ for $t\geq t_1+\frac{1}{\nu \lambda_1}$. Therefore, by employing (\ref{in}) along with (\ref{change}), we deduce
\begin{align*}
|B(u,\bar u)|  \leq c|u|^{\frac{1}{2}}\|u\|^{\frac{1}{2}}
\|\bar u\|^{\frac{1}{2}} |A\bar u|^{\frac{1}{2}}
\leq \frac{c}{\alpha} |u| \|u\|^{\frac{1}{2}} \|\bar u\|^{\frac{1}{2}} \leq \frac{c}{\alpha} \rho_3 \rho_1^{\frac{1}{2}} \|u\|^{\frac{1}{2}},
\text{\;\;for\;\;} t\geq t_1+\frac{1}{\nu \lambda_1}.
\end{align*}

As a result, for $t\geq t_1+\frac{1}{\nu \lambda_1}$,
\begin{align*}
\frac{d}{dt}\|u\|^2 &\leq \frac{c}{\nu \alpha^2} \rho_3^2 \rho_1 \|u\| +\frac{2}{\nu} |f|^2.
\end{align*}
To obtain a uniform bound for $\|u(t)\|$, we integrate between $s$ and $t+\frac{1}{\nu \lambda_1}$, for $t_1+\frac{1}{\nu \lambda_1}\leq t \leq s \leq t+\frac{1}{\nu \lambda_1}$,
\begin{align*}
\|u(t+\frac{1}{\nu \lambda_1})\|^2 \leq \|u(s)\|^2 + \frac{c}{\nu \alpha^2} \rho_3^2 \rho_1 \int_t^{t+\frac{1}{\nu \lambda_1}}  \|u(s)\| ds   +  \frac{2}{\nu^2 \lambda_1}  |f|^2.
\end{align*}
Then, using Cauchy-Schwarz and integrating with respect to $s$ between $t$ and $t+\frac{1}{\nu \lambda_1}$ yield
\begin{align*}
\frac{1}{\nu \lambda_1}\|u(t+\frac{1}{\nu \lambda_1})\|^2 \leq   \int_t^{t+\frac{1}{\nu \lambda_1}}\|u(s)\|^2 ds
+ \frac{c}{\nu^{\frac{5}{2}} \alpha^2  \lambda_1^{\frac{3}{2}}} \rho_3^2 \rho_1  \left(\int_t^{t+\frac{1}{\nu \lambda_1}}\|u(s)\|^2 ds\right)^{\frac{1}{2}}
+  \frac{2}{\nu^3  \lambda_1^2} |f|^2,
\end{align*}
for $t\geq t_1+\frac{1}{\nu \lambda_1}$.
Now we ought to find a bound for $\int_t^{t+\frac{1}{\nu \lambda_1}}\|u(s)\|^2 ds$. Indeed, integrating (\ref{L2-2}) from $t$ to $t+\frac{1}{\nu \lambda_1}$ for $t\geq t_1+\frac{1}{\nu \lambda_1}$ gives
\begin{align*}
\int_t^{t+\frac{1}{\nu \lambda_1}} \|u(s)\|^2 ds
&\leq \frac{1}{\nu}\left(|u(t)|^2+ \frac{c}{\nu} \rho_1^2 \int_t^{t+\frac{1}{\nu \lambda_1}} |u(s)|^2 ds +\frac{2}{\nu^2 \lambda_1} |A^{-\frac{1}{2}}f|^2\right) \notag\\
&\leq C_1:=\frac{1}{\nu}\left(\rho_3^2 + \frac{c}{\nu}\rho_1^2 C_0 + \frac{2}{\nu^2 \lambda_1} |A^{-\frac{1}{2}}f|^2\right)  ,
\end{align*}
where we have used (\ref{L2-3}) and the fact that $|u(t)|\leq \rho_3$ provided $t\geq t_1+\frac{1}{\nu \lambda_1}$.

Finally, we conclude
\begin{align*}
\|u(t)\|^2 \leq  \tilde \rho^2:= \nu \lambda_1  C_1 +   \frac{c}{\nu^{\frac{3}{2}} \alpha^2  \lambda_1^{\frac{1}{2}}    } \rho_3^2 \rho_1 C_1^{\frac{1}{2}} +\frac{2}{\nu^2 \lambda_1} |f|^2,
\text{\;\;for\;\;} t\geq t_1+\frac{2}{\nu \lambda_1}.
\end{align*}
This shows the solution $u(t)$ enters of a ball in $V$ of radius $\tilde \rho$ for $t\geq t_1+\frac{2}{\nu \lambda_1}$.

Also, recall $\rho_0 \sim \nu^{-1}$, $\rho_1 \sim \nu^{-1}$, $\rho_3 \sim \nu^{-3}$, then by (\ref{L2-3}) one has $C_0 \sim \nu^{-3}$, and thus we see that $C_1 \sim \nu^{-7}$. Hence, $\tilde \rho \sim  \nu^{-6}.$

\bigskip

\subsection{Existence of an inertial manifold} \label{exist}
From energy estimates established in section \ref{modsec}, we see that for positive time $t$, one has $u(t) \in V$ because of the parabolic nature of the equation, and for sufficiently large time $t\geq t_1+\frac{2}{\nu \lambda_1}$, the solution $u(t)$ enters a ball in $V$ of radius $\tilde \rho$. So, without loss of generality, as far as inertial manifold is concerned, which is a long-time behavior, we assume the initial data $u_0 \in V$.

We set $\mathcal R(u):=B(u,\bar u)$. Then the equation (\ref{bardina2}) takes the form
\begin{align}    \label{r'}
u_t+\nu Au+\mathcal R(u)=f.
\end{align}

 Recall that the nonlinear term $B(\bar v,\bar v)=P_{\sigma} (\bar v \cdot \nabla)\bar v$ in the simplified Bardina model (\ref{bar1}) is locally Lipschitz from $H$ to $H$, which is a condition for (\ref{bar1}) possessing an inertial manifold. However, $\mathcal R(u)$ does not have this property, since it is not a mapping from $H$ to $H$. However, we will be able to show that $\mathcal R$ is locally Lipschitz continuous from $V$ to $V$. To see this, we calculate
\begin{align} \label{estR}
\|\mathcal R(u)\|=|A^{\frac{1}{2}} B(u,\bar u)|
&\leq | B(A^{\frac{1}{2}}u,\bar u)|+| B(u, A^{\frac{1}{2}} \bar u)|\notag\\
&\leq \|u\| |\nabla \bar u|_{L^{\infty}}+ c |u|^{\frac{1}{2}}\|u\|^{\frac{1}{2}} |A \bar u|^{\frac{1}{2}} |A^{\frac{3}{2}}  \bar u|^{\frac{1}{2}} \notag\\
&\leq c\|u\| \|\bar u\|^{\frac{1}{2}}  |A^{\frac{3}{2}} \bar u|^{\frac{1}{2}}
+c |u|^{\frac{1}{2}}\|u\|^{\frac{1}{2}} |A \bar u|^{\frac{1}{2}} |A^{\frac{3}{2}}  \bar u|^{\frac{1}{2}} \notag\\
&\leq c \lambda_1^{-\frac{1}{2}} \|u\| |A^{\frac{3}{2}} \bar u|  \notag\\
&\leq c \lambda_1^{-\frac{1}{2}} \alpha^{-2} \|u\|^2.
\end{align}
Note that, throughout the above calculation, we have employed (\ref{in}), and Agmon's inequality in 2D: $|\phi|_{L^{\infty}}\leq c|\phi|^{\frac{1}{2}} |A\phi|^{\frac{1}{2}}$, as well as (\ref{change}),
where $c$ is a positive constant.

This shows that $\mathcal R$ is a mapping from $V$ to $V$. By similar computation, we deduce, for $u_1$, $u_2\in V$:
\begin{align}    \label{lip}
\|\mathcal R(u_1)-\mathcal R(u_2)\| \leq c \lambda_1^{-\frac{1}{2}} \alpha^{-2}    (\|u_1\|+\|u_2\|)\|u_1-u_2\|,
\end{align}
that is, $\mathcal R: V\rightarrow V$ is locally Lipschitz continuous.

Recall in the subsection \ref{sec5}, we have shown that $|u(t)| \leq \tilde \rho$ for sufficiently large time $t\geq t_1+\frac{2}{\nu \lambda_1}$. As in \cite{fsr,fst}, in order to avoid certain technical difficulties for large values of $\|u\|$, resulting from the nonlinearity, we truncate the nonlinear term outside the ball of radius $2\tilde \rho$ in $V$      by a smooth cutoff function $\theta: \mathbb R^+ \rightarrow [0,1]$ with $\theta(s)=1$ for $0\leq s \leq 1$, $\theta(s)=0$ for $s\geq 2$, and $|\theta'(s)|\leq 2$ for $s\geq 0$. Define $\theta_{\tilde \rho}(s)=\theta(s/\tilde \rho)$ for $s\geq 0$. We consider the following ``prepared" equation, which is a modification of (\ref{r'}):
\begin{equation} \label{rr'}
u_t+\nu Au+\theta_{\tilde \rho}(\|u\|)(\mathcal R(u)-f)=0.
\end{equation}

Since $\mathcal R: V\rightarrow V$ is locally Lipschitz, by similar calculation as in subsection \ref{mani}, it can be shown that the truncated nonlinearity $\mathcal F(u):=\theta_{\tilde \rho}(\|u\|)\mathcal R(u)$ is globally Lipschitz continuous with Lipschitz constant $\mathcal L:=c\tilde \rho \lambda_1^{-\frac{1}{2}} \alpha^{-2}$.

Now, for $\gamma>0$ and $N\in \mathbb N$, we define the cone in the product space $V\times V$:
\begin{align*}
\tilde \Gamma_{N,\gamma}:=\left\{ \begin{pmatrix}u_1\\u_2\end{pmatrix} \in V \times V: \|Q_N(u_1-u_2)\| \leq \gamma \|P_N(u_1-u_2)\|   \right \}.
\end{align*}
The following result states the equation (\ref{rr'}) possesses the \emph{strong squeezing property}:

\begin{proposition} \label{prop2}
 Let $u_1$ and $u_2$ be two solutions of (\ref{rr'}).
Then (\ref{rr'}) satisfies the following properties:
\begin{enumerate}
\item The \emph{cone invariance property}: Assume that $N$ is large enough such that the spectral gap condition $\lambda_{N+1}-\lambda_N > \frac{\mathcal L (\gamma+1)^2}{\nu \gamma} $ holds. If $\begin{pmatrix}u_1(t_0)\\u_2(t_0)\end{pmatrix}\in \tilde \Gamma_{N,\gamma} $ for some $t_0\geq 0$,
then $\begin{pmatrix}u_1(t)\\u_2(t)\end{pmatrix}\in \tilde \Gamma_{N,\gamma} $ for all $t\geq t_0$;

\item The \emph{decay property}: Assume that $N$ is sufficiently large such that $\lambda_{N+1}>\nu^{-1}\mathcal L\left(\frac{1}{\gamma}+1\right)$. If $\begin{pmatrix}u_1(t)\\u_2(t)\end{pmatrix} \not\in \tilde \Gamma_{N,\gamma} $ for $0\leq t \leq  T$, then
\begin{align*}
\|Q_N (u_1(t)-u_2(t)) \| \leq \|Q_N(u_1(0))-u_2(0)\|e^{-\beta_N t}, \text{\;\;for\;\;} 0\leq t\leq T,
\end{align*}
where $\beta_N  := \nu \lambda_{N+1}-\mathcal L\left(\frac{1}{\gamma}+1\right)>0$.
\end{enumerate}
\end{proposition}

\begin{proof}
See the appendix.
\end{proof}

Note that the spectral gap condition is satisfied for sufficiently large $N$, by virtue of Theorem \ref{rich}. Consequently the strong squeezing property holds for the equation (\ref{rr'}). Then, according to Theorem \ref{thm-g} which concerns the existence of an inertial manifold, we have the following result.

\begin{theorem}
The ``prepared" equation (\ref{rr'}) of the modified-Leray-$\alpha$ model possesses an $N$-dimensional inertial manifold $\mf M$ in $V$, i.e., the solution $S(t)u_0$ of (\ref{rr'}) approaches the invariant Lipschitz manifold $\mf M$ exponentially in $V$. Furthermore, the following exponential tracking property holds: for any $u_0\in V$, there exists a time $\tau\geq 0$ and a solution $S(t)\varphi_0$ on the inertial manifold $\mf M$ such that
\begin{align*}
\|S(t+\tau)u_0-S(t)\varphi_0\| \leq C e^{-\beta_N t},
\end{align*}
where $\beta_N$ is defined in Proposition \ref{prop2}, and the constant $C$ depends on $\|S(\tau)u_0\|$ and $\|\varphi_0\|$.
\end{theorem}

\smallskip

\begin{remark}
Concerning the Leray-$\alpha$ model (\ref{Leray}), the nonlinearity is $(\bar w \cdot \nabla)w$ and clearly there is a loss of derivative. It can be shown that the operator $\tilde R(v):=B(\bar v,v)=P_{\sigma} (\bar v \cdot \nabla)v$ is Lipschitz continuous from $V$ to $H$ in 2D. As far as inertial manifold is concerned, this produces the similar difficulty as what we face for the 2D NSE. Indeed, under such scenario, using the classical theory, the existence of an inertial manifold requires a stronger gap condition: $\lambda_{j+1}^{\frac{1}{2}}-\lambda_{j}^{\frac{1}{2}}$ must be sufficiently big, which only holds for very large viscosity $\nu$ (see, e.g. \cite{Robin}). But our main interest lies in fluid flow with small viscosity, which is the situation when turbulence occurs, so a result valid for only large $\nu$ is of no account.
\end{remark}

\section{APPENDIX}
We present the proof of Propositions \ref{prop1} and \ref{prop2} for the sake of completion. Since the proof of these two propositions are similar, we only show Proposition \ref{prop2}.
\begin{proof}
The method of the proof is standard (see, e.g. \cite{fst}). Assume $u_1$ and $u_2$ are two solutions of (\ref{rr'}). To show the cone invariance property (i), it is sufficient to show
$\begin{pmatrix}u_1(t)\\u_2(t)\end{pmatrix}$ can not pass through the boundary of the cone if the dynamics starts inside the cone. More precisely, we shall show $\frac{d}{dt} (\|Q_N(u_1(t)-u_2(t))\|-\gamma \|P_N(u_1(t)-u_2(t))\|)<0$ whenever $\begin{pmatrix}u_1(t)\\u_2(t)\end{pmatrix}\in \partial \tilde\Gamma_{N,\gamma}$, where $\partial \tilde \Gamma_{N,\gamma}$ stands for the boundary of the cone $\tilde \Gamma_{N,\gamma}$.

Recall $\mathcal F(u)=\theta_{\tilde \rho}(\|u\|)\mathcal R(u)$. Then by the equation (\ref{rr'}),
\begin{align*}
\partial_t(u_1-u_2)+\nu A (u_1-u_2) + \mathcal F(u_1)-\mathcal F(u_2)=0.
\end{align*}
By setting $p=P_N (u_1-u_2)$ and $q=Q_N (u_1-u_2)$, we obtain
\begin{align}
&p_t+\nu A p + P_N(\mathcal F(u_1)-\mathcal F(u_2))=0  \label{low}\\
&q_t+\nu A q + Q_N(\mathcal F(u_1)-\mathcal F(u_2))=0.  \label{high}
\end{align}

We take the scalar product of (\ref{low}) with $A p$,
\begin{align*}
\frac{1}{2}\frac{d}{dt} \|p\|^2 + \nu |A p|^2 +    \left( P_N(\mathcal F(u_1)-\mathcal F(u_2)) ,A p \right)=0.
\end{align*}
Thus by the global Lipschitz continuity of $\mathcal F$, we have
\begin{align} \label{div}
\frac{1}{2}\frac{d}{dt} \|p\|^2  \geq -\nu \lambda_N \|p\|^2 - \|\mathcal F(u_1)-\mathcal F(u_2)\|   \|p\|
\geq -\nu \lambda_N \|p\|^2 - \mathcal L \|u_1-u_2\| \|p\|.
\end{align}
Without loss of generality, we can assume $\|p(t)\|>0$. (Otherwise, if $\|p(t^*)\|=0$ for some $t^*$, then since we consider the boundary of the cone, we can assume $\|q(t^*)\|=\gamma \|p(t^*)\|=0$, and thus $u_1(t^*)=u_2(t^*)$. By the uniqueness of solutions, we obtain $u_1(t)=u_2(t)$ for all $t\geq t^*$, and the cone invariance property follows.)   Now we can divide both sides of (\ref{div}) by $\|p(t)\|$, so
\begin{align} \label{ana'}
\frac{d}{dt}  \|p\| \geq -\nu \lambda_N \|p\| - \mathcal L \|u_1-u_2\|.
\end{align}

Analogously, by taking the scalar product of (\ref{high}) with $Aq$, we can deduce
\begin{align}  \label{ana}
\frac{d}{dt} \|q\|  \leq -\nu \lambda_{N+1} \|q\| + \mathcal L \|u_1-u_2\|.
\end{align}

Multiplying (\ref{ana'}) with $\gamma$ and subtracting the result from (\ref{ana}), we infer, by using the fact $p+q=u_1-u_2$,
\begin{align*}
\frac{d}{dt} (\|q\|-\gamma \|p\|) \leq \nu (\lambda_N \gamma \|p\|-\lambda_{N+1} \|q\| )   + \mathcal L(\gamma+1) (\|p\|+\|q\|).
\end{align*}
So whenever $\|q(t)\|=\gamma \|p(t)\|$, i.e. $\begin{pmatrix}u_1(t)\\u_2(t)\end{pmatrix}\in \partial \tilde \Gamma_{N,\gamma}$, we have
\begin{align*}
\frac{d}{dt} (\|q\|-\gamma \|p\|) \leq \left(\nu (\lambda_N-\lambda_{N+1}) +\mathcal L \frac{(\gamma+1)^2}{\gamma} \right) \|q\|  <0,
\end{align*}
due to our assumption $\lambda_{N+1}-\lambda_N > \frac{\mathcal L (\gamma+1)^2}{\nu \gamma}   $.

To show the decay property (ii), we assume $\begin{pmatrix}u_1(t)\\u_2(t)\end{pmatrix} \not\in \tilde \Gamma_{N,\gamma} $ for $0\leq t\leq T$, then $\|q(t)\|>\gamma\| p(t)\|$ for $0\leq t\leq T$, and we see from (\ref{ana}) that
\begin{align*}
\frac{d}{dt} \|q\|  \leq -\nu \lambda_{N+1} \|q\| + \mathcal L (\|p\|+\|q\|)
\leq -\left[\nu \lambda_{N+1}  -\mathcal L \left( \frac{1}{\gamma}   +1    \right)  \right] \|q\|
=-\beta_N  \|q\|,
\end{align*}
for $0\leq t\leq T$, where $\beta_N:=\nu \lambda_{N+1}  -\mathcal L \left( \frac{1}{\gamma}   +1    \right)$. By Gronwall's inequality, one has
\begin{align*}
\|q(t)\|\leq e^{-\beta_N t} \|q(0)\|, \text{\;\;for\;\;}  0\leq t\leq T.
\end{align*}
\end{proof}

\noindent {\bf Acknowledgment.} This work was supported in part by the NSF grants DMS--1109640 and DMS--1109645.

\end{document}